%% file: LowerBound_LlosaTessera-final.tex
\documentclass[a4paper,11pt,reqno]{amsart}

\usepackage{amsmath}
\usepackage{amsfonts}
\usepackage{amssymb}
\usepackage{amsthm}
\usepackage[shortlabels]{enumitem}
\usepackage{graphicx}
\usepackage{color}
\usepackage{dsfont}
\usepackage[latin1]{inputenc}
\usepackage{MnSymbol}
\usepackage{enumitem}
\usepackage{extpfeil}
\usepackage{mathtools}
\usepackage{url}
\usepackage[margin=1.2in]{geometry}
\usepackage{fancyhdr}
\usepackage[all,cmtip]{xy}

\numberwithin{equation}{section}

\title{Residually free groups do not admit a uniform polynomial isoperimetric function}

\author{Claudio Llosa Isenrich}
\address{Max Planck Institute for Mathematics, Vivatsgasse 7, 53111 Bonn, Germany}
\email{llosa@mpim-bonn.mpg.de}
\author{Romain Tessera}
\address{Institut de Math\'ematiques de Jussieu-PRG, Universit\'e Paris-Diderot, CNRS, Case 7012, 75205 Paris Cedex 13, France}
\email{romatessera@gmail.com}

\thanks{The first author was supported by a public grant as part of the FMJH. The second author was supported by the grant ANR-14-CE25-0004 ``GAMME''. \newline \indent The authors would like to thank the anonymous referees for their helpful comments and suggestions.}
\keywords{Residually free groups, Dehn functions}
\subjclass[2010]{20F65, (20F05, 20F69)}

\begin{document}

\input{commands}

\begin{abstract}
We show that there is no uniform polynomial isoperimetric function for finitely presented subgroups of direct products of free groups, by producing a sequence of subgroups $G_r\leq F_2^{(1)} \times \dots \times F_2^{(r)}$ of direct products of 2-generated free groups with Dehn functions bounded below by $n^{r}$. The groups $G_r$ are obtained from the examples of non-coabelian subdirect products of free groups constructed by Bridson, Howie, Miller and Short. As a consequence we obtain that residually free groups do not admit a uniform polynomial isoperimetric function.
\end{abstract}

\maketitle

\section{Introduction}

Subgroups of direct products of finitely many free groups (for short SPF groups) form a special class of residually free groups that have attracted a lot of attention due to their interesting topological finiteness properties \cite{BriHowMilSho-02, BriHowMilSho-13}. Indeed, the first examples of groups admitting a classifying space with finite $k$-skeleton, but no classifying space with finite $(k+1)$-skeleton, for $k\geq 2$ -- the Stallings--Bieri groups -- belong to that class \cite{Sta-63,Bie-76}. 

Let us say that a subgroup $G\leq G_1\times \dots \times G_r$ of a finite product of groups $G_i$ is VCA (for virtually coabelian) if there are finite index subgroups $G_{i,0}\leq G_i$ and a surjective homomorphism $\phi: G_{1,0}\times \dots \times G_{r,0}\to \ZZ^k$, such that $\ker(\phi)\leq G$ is a finite index subgroup of $G$.
It turns out that all SPF groups with strong enough finiteness properties are VCA \cite[Corollary 3.5]{Kuc-14}. While not all SPF groups are VCA, they all contain 
 finite index subgroups which are iterated fibre products over nilpotent groups \cite{BriHowMilSho-13}. 
 
SPF groups which have the VCA property have been the focus of various results in the area, shedding light on many aspects of their nature (for instance \cite{BriMil-09,Dis-08,Koc-10, Kuc-14}). In contrast we know much less about SPF groups which are not VCA. In particular, while we know that SPF groups which are not VCA must satisfy strong constraints, there is currently only one class of known examples of this kind. They were constructed by Bridson, Howie, Miller and Short \cite{BriHowMilSho-13} and will play a key role in this work.

A way to study topological properties of groups from a quantitative point of view is to estimate their filling invariants, and in particular their Dehn functions. 
Recall that the Dehn function $\delta_G(n)$ of a finitely presented group $G=\left\langle X \mid R\right\rangle$  is defined as the number of conjugates of relations from $R$ needed to detect if a word in $X$ of length at most $n$ represents the trivial word. Dehn functions have important connections to the solvability of the word problem and to isoperimetric functions in Riemannian geometry, highlighting the importance of understanding this basic invariant \cite{Bri-02,BurTab-02}.

The study of Dehn functions of SPF groups was initiated by Gersten who proved that the Dehn functions of the Stallings--Bieri groups admit a quintic upper bound \cite{Ger-95}. This bound was improved to a cubic bound in \cite{BauBriMilSho-97}. In \cite{Bri-99} Bridson argued that in fact the Dehn functions of Stallings--Bieri groups are quadratic. There was a flaw in his argument, but it was subsequently proved that the assertion that they are quadratic is correct \cite{DisEldRilYou-09, CarFos-17}. 

Bridson conducted a general study of the Dehn functions of cocyclic subgroups of products of groups \cite{Bri-01}. He showed that finitely presented cocyclic subgroups of a product of groups $G_1\times \dots \times G_r$ admit a polynomial isoperimetric function, if the $\delta_{G_i}$ are all polynomial \cite{Bri-01}. In particular it follows directly from his work and the fact that limit groups are CAT(0) \cite{AliBes-06} that all cocyclic subgroups of products of limit groups admit a polynomial isoperimetric function. In \cite{Dis-08} Dison pursued a systematic study of the Dehn functions of coabelian subgroups of direct products of limit groups. He showed that when the group has strong enough finiteness properties, then its Dehn function is polynomial. He also proved that a large class of full subdirect products of limit groups admit a sextic isoperimetric function. More generally for Bestvina--Brady groups, which are a generalization of the Stallings--Bieri groups to subgroups of Right Angled Artin groups, he obtains a quartic bound on their Dehn functions \cite{Dis-08-II}. 

Dison also provides the first example of an SPF group which does not admit a quadratic (or linear) isoperimetric function: he actually shows that his example satisfies a cubic lower bound and a sextic upper bound on its Dehn function \cite{Dis-08-II,Dis-08}. 

These results naturally led Dison to ask the following question:
\begin{question}\cite{Dis-08}\label{que:Dison}
Is there a uniform polynomial bound $p(n)$ such that $\delta_G(n)\preccurlyeq p(n)$ for all SPF groups $G$? 
\label{qnDison}
\end{question}

The purpose of this note is to show that the answer to Question \ref{que:Dison}  is negative. 

\begin{theorem}
 For every $r\geq 3$ there is a finitely presented subgroup $G_r\leq F_2^{(1)}\times \dots \times F_2^{(r)}$ with $\delta_{G_r}(n) \succcurlyeq n^r$.
 \label{thmMain}
\end{theorem}

The groups $G_r$ in Theorem \ref{thmMain} are explicitly described as quotients of the groups constructed by Bridson, Howie, Miller and Short in \cite[Section 4]{BriHowMilSho-13}, which themselves are conilpotent subgroups of direct products of free groups. As a direct consequence we obtain:
\begin{corollary}
\label{corMainThm}
The class of finitely presented residually free groups does not admit a uniform polynomial isoperimetric function.
\end{corollary}

Let us close this introduction with a few questions that arise naturally from this result.
We first recall the following non-uniform version of Dison's question:
\begin{question}\cite{Dis-08}
 Does every finitely presented subgroup of a product of free (or limit) groups admit a polynomial isoperimetric function?
\end{question}

We can break this question into the two following more specific ones:
\begin{question}
Does a finitely presented subgroup of a product of $r$ free groups have isoperimetric function $n^r$?
\end{question}

\begin{question}
 Is there a uniform isoperimetric function if we restrict to subdirect products of a fixed conilpotency class?
\end{question}
\newpage
\section{Background}

Let $G$ be a finitely presented group and let $\mathcal{P}=\left\langle X \mid R\right \rangle$ be a finite presentation. For a word $w(X)=x_1\dots x_n$, with $x_i\in X^{\pm 1}$, let $l(w(X))=n$ be its \textit{word length}, and for an element $g\in G$ let $|g|_G=\mathrm{dist}_{Cay(G,X)}(1,g)$ be its distance from the identity in the Cayley graph $Cay(G,X)$. We say that a word $w(X)$ is \textit{null-homotopic} if it represents the trivial element in $G$. The area of a null-homotopic word is
{\small\[
 \mathrm{Area}_{\mathcal{P}}(w(X))= \mathrm{min}\left\{ k \mid w(X)=_{Free(X)} \prod_{i=1}^k \theta_i(X) r_i \theta_i(X)^{-1}, r_i\in R^{\pm 1}, \theta_i(X) \mbox{ words in } X \right\}.
\]}

The Dehn function of $G$ (with respect to the presentation $\mathcal{P}$) is then defined as
\[
 \delta_G(n)=\mathrm{max}\left\{ \mathrm{Area}_{\mathcal{P}}(w(X)) \mid w(X) \mbox{ null-homotopic, }~ l(w(X))\leq n\right\}.
\]

For non-decreasing functions $f,g: \NN \to \RR_{\geq 0}$ we say that $f$ is \textit{asymptotically bounded} by $g$ and write $f\preccurlyeq g$ if there is a constant $C\geq 1$ such that $f(n) \leq C g(C n) +Cn +C$ for all $n\in \NN$. We say that $f$ is \emph{asymptotically equal} to $g$ and write $f\asymp g$ if $f\preccurlyeq g \preccurlyeq f$. Note that $\mathrm{Area}_{\mathcal{P}}$ and the Dehn function of $G$ are well-defined up to asymptotic equivalence, with respect to changes of presentation $\mathcal{P}$. This justifies the notations $\delta_G$ and $\mathrm{Area}_G$, which we will make frequent use of wherever this does not lead to confusion. We shall also use the notion $f\lesssim g$ if there exists a constant $C\geq 1$ such that $f\leq Cg$, and $ f\simeq g$ if both $f\lesssim g$ and $g\lesssim f$.  $|\cdot|_G$ is independent of the choice of finite presentation up to $\simeq$. We will frequently use this fact without explicit mention when working with $\lesssim$ and $\simeq$.

\begin{definition}
 For $r\geq 1$, $G_1\times \dots \times G_r$ a direct product of $r$ groups, $r\geq k \geq 1$,  and $1\leq i_1< \dots <i_k\leq r$, let $p_{i_1,\dots,i_k}: G_1\times \dots \times G_r \to G_{i_1}\times \dots \times G_{i_k}$ be the canonical projection. We say that a subgroup $H\leq G_1\times \dots \times G_r$ has the \textit{VSP property} (virtual surjection to pairs property) if $p_{i_1,i_2}(H)\leq G_{i_1}\times G_{i_2}$ is a finite index subgroup for all $1\leq i_1<i_2\leq r$.
 
 We further say that $H\leq G_1\times \dots \times G_r$ is \textit{subdirect} if $p_i(H)=G_i$ for all $1\leq i \leq r$ and \textit{full} if $H\cap G_i \left(:= H \cap(1\times \dots \times 1\times G_i\times 1 \times \dots\times 1)\right)\neq 1$ for all $1\leq i \leq r$.
\end{definition}

A group $G$ is called a \textit{limit group} (or \textit{fully residually free}) if for every finite subset $S\subset G$ there is a homomorphism $\phi: G\to F_2$ such that the restriction $\phi|_S$ is injective. We call $G$ \textit{residually free} if for every $g\in G\setminus \left\{1\right\}$ there is a homomorphism $\phi: G\to F_2$ with $\phi(g)\neq 1$. A finitely generated group $G$ is residually free if and only if $G$ admits an embedding in a finite product of limit groups \cite{BauMyaRem-99} (see also \cite{BriHowMilSho-13}), emphasizing the importance of studying subgroups of products of limit groups.  As mentioned in the introduction, SPFs are an important special class of residually free groups, and in this paper we will mostly restrict our attention to them. However, our results have implications for the class of general residually free groups in the form of Corollary \ref{corMainThm}, which is the motivation for defining them here.

For a group $G$ we denote by $\gamma_{c}(G)$ the $c$-th term of the lower central series of $G$, recalling that it is defined inductively by $\gamma_1(G)=G$ and $\gamma_{c+1}(G)=\left[G,\gamma_c(G)\right]$. A key result on finitely presented subgroups of direct products of limit groups is the following:
\begin{theorem}
 Let $H\leq G_1\times \dots \times G_r$ be a full subdirect product that satisfies the VSP property. Then $H$ is finitely presented and there are finite index subgroups $G_{i,0}\leq G_i$ such that $\gamma_{r-1}(G_{i,0})\leq H$. 
 
 Conversely, if $G_1,\dots, G_r$ are finitely generated limit groups and $H\leq G_1\times \dots \times G_r$ is a finitely presented full subdirect product then $H$ has the VSP property.
 \label{thmBHMS13}
\end{theorem}
\begin{proof}
 This is a direct consequence of Theorem A, Theorem D and Proposition 3.2 in \cite{BriHowMilSho-13}. 
 \end{proof}

\section{Large distortion or large Dehn function}
\label{secDehnDist}
In this section we want to explain a generalisation of Proposition 4.2 in \cite{LloTes-18} to asymmetric fibre products in a product of non-hyperbolic groups.
Let $G_1$, $G_2$ be groups and let $\phi_1:G_1\to Q$ and $\phi_2: G_2\to Q$ be surjective homomorphisms. The (asymmetric) fibre product $P\leq G_1\times G_2$ is the subgroup defined by
$$P=\left\{\left(g_1,g_2\right)\in G_1\times G_2\mid \phi_1(g_1)=\phi_2(g_2)\right\}.$$

Assume that $G_1$ and $G_2$ are also finitely presented. We equip these groups with presentations $G_1 = \left\langle X_1\mid R_1 \right\rangle$, $G_2=\left\langle X_2\mid S_2 \right\rangle$ and  $Q=\left\langle X_Q \mid T_Q \right\rangle$ such that the families $X_1=\{a_1,b_1,\ldots\}$, $X_2=\{a_2,b_2,\ldots\}$ and $X_Q=\{a_Q,b_Q,\ldots\}$ are in bijection with a given family $X=\{a,b,\ldots\}$. We shall assume that these bijections induce morphisms
$j_1:F_X\to G_1$, $j_2:F_X\to G_2$, $j_Q:F_X\to Q$ satisfying $j_Q=\phi_1\circ j_1=\phi_2\circ j_2:F_X\to Q$. Using the identification of the generating families $X_1, X_2$ and $X_Q$ with $X$, we define (in the obvious way) the subsets $R,S$ and $T$ of $F_X$.  Note that since $Q$ is a quotient of both $G_1$ and $G_2$, we can  assume on increasing $T$ if necessary that it contains $R$ and $S$. Denote by $X_{\Delta}=\left\{(j_1(x),j_2(x)\mid x\in X\right\}\subset X_1\times X_2$.  Finally to every $t\in T$ we associate a letter $t'$, whose set is denoted by $T'$. Let  $T_1=\left\{\left(j_1(t),1\right)\mid t\in T\right\}$. 

It follows from the so-called asymmetric 0--1--2 Lemma that $X_{\Delta} \cup T_1$ is a finite generating set for $P$. This naturally defines $P$ as a quotient of the free group $F_{X\cup T'}$. While this lemma is well-known to experts, there does not seem to be a standard reference for it. We will thus include a proof here. We also remark that the symmetric version, where $G_1=G_2$ and $\phi_1=\phi_2$, can for instance be found in \cite[Lemma 2.1]{BriGru-04}. Note that the proof of the 0--1--2 Lemma only requires $G_1$ and $G_2$ to be finitely generated. We will thus deviate slightly from the above notation for the statement and proof of this result, in order to provide it in its full level of generality: we will assume that $Q=\left\langle X\mid T \right\rangle$ is finitely presented without requiring that $R$ and $S$ are contained in the finite set $T$ and define all other sets, groups and morphisms as above.

\begin{lemma}[{Asymmetric 0--1--2 Lemma}]
 Let $G_1=\left\langle X_1\right\rangle$ and $G_2=\left\langle X_2\right\rangle$ be finitely generated groups, let $Q=\left\langle X_Q \mid T_Q\right\rangle$ be finitely presented and let $\phi_i:G_i\to Q$ be epimorphisms. Then the fibre product $P\leq G_1\times G_2$ is generated by $X_{\Delta}\cup T_1$. In particular, $P$ is finitely generated.
\end{lemma}
\begin{proof}
 By definition $X_{\Delta}\cup T_1\subset P$. We thus only need to check that it is a generating set. Let $g=\left(g_1,g_2\right)\in P$ and let $w(X_1,X_2)=u(X_1)\cdot v(X_2)$ be a word representing $g$ in $G_1\times G_2$.
 
 The word $w(X_1,X_2) \cdot v(X_{\Delta})^{-1}=\widetilde{w}(X_1)$ represents an element of the form $(h,1)$ in $G_1\times G_2$, which is contained in $P$. Thus $\phi_1(h)=1$, implying that $\widetilde{w}(X_Q)$ is freely equal to a product of conjugates of relations from $T_Q$ and their inverses by words in $X_Q$. Hence, $\widetilde{w}(X_1)$ satisfies 
 $$\widetilde{w}(X_1)=\prod_{j=1}^k (\theta_j(X_1),1) \cdot (j_1(t_j),1)^{\pm 1} \cdot (\theta_j(X_1)^{-1},1) = \prod_{j=1}^k \theta_j(X_{\Delta}) \cdot (j_1(t_j),1)^{\pm 1} \cdot \theta_j(X_{\Delta})^{-1} $$ 
 for words $\theta_j$ and relations $t_j\in T$. This completes the proof.
\end{proof}

From now on we will always assume that $G_1$ and $G_2$ are finitely presented and use the notation we introduced at the beginning of this section. The main result of this section is the following generalisation of \cite[Proposition 4.2]{LloTes-18}. We will combine this result with the existence of certain conilpotent subgroups of products of $r\geq 3$ free groups proved in \cite{BriHowMilSho-13} to deduce Theorem \ref{thmMain}.

\begin{proposition}\label{prop:areaDehn}
 Let $G_1=\left\langle X_1 \mid R_1\right\rangle$, $G_2=\left\langle X_2\mid S_2\right\rangle$, $Q=\left\langle X_Q\mid T_Q\right\rangle$ and $P=\left\langle X_{\Delta}\cup T_1\right\rangle$ be as above.  Let $h=(g,1)\in P\cap \left(G_1\times \left\{1\right\}\right)$. Let $v\in F_X$ be such that $g=j_1(v)$, and let $C=\max \{|t|_X;\; t\in T\}$. Then 
 \[\mathrm{Area}_Q(v)\leq |h|_{P}+\delta_{G_2}( |h|_{P})+\delta_{G_1}( C|h|_{P}+|v|_{F_X}).\]
\end{proposition}
\begin{proof}[Proof of Proposition \ref{prop:areaDehn}]
We let $w=w(X,T') \in F_{X\cup T'}$ be a reduced word of length $|h|_P$ such that $h=w(X_{\Delta},T_1)$, where $w(X_{\Delta},T_1)$ means the element of $P$ obtained by substituting letters of $w$ in $X$ and $T'$ by the corresponding elements in $X_{\Delta}$ and $T_1$. 
\begin{clai}\label{clai:w(X,1)}
The element $w(X,1)\in F_X$ is a product of conjugates of elements of $S$ whose number $k$ is at most $\delta_{G_2}(|h|_P)$.
\end{clai}
\begin{proof}
Projecting $w$ in $P$, and using the product structure in $G_1\times G_2,$ we obtain
\[w(X_{\Delta},T_1)=w(X_1,T_1)w(X_2,1).\] So by projecting to $G_2$, we deduce that $w(X_2,1)$ maps to the trivial element in $G_2$, so we are done.
\end{proof}

Let $q$ be the number of letters from $T'$ in $w$, and let $\pi: F_{X\cup T'}\to F_{X}$ be the group morphism mapping $T'$ to $T$. 
\begin{clai}\label{clai:k+pT}
The word $\pi(w)$ is freely equal to a product of $q+k$ conjugates of elements of $T$, and thus $ \mathrm{Area}_Q(\pi(w))\leq q+k$.   
\end{clai}
\begin{proof}
By moving all letters of $T'$ to the right, we can write $w(X,T')=w(X,1)w'$ in $F_{X\cup T'}$, where $w'$ is a product of $q$ conjugates of letters in $T'$. On the other hand by Claim \ref{clai:w(X,1)}, $w(X,1)$ can be written as a product of $k$ conjugates of elements of $S\subset T$. \end{proof}
We now consider the element $v\pi(w)^{-1}\in F_X$, which by construction is mapped to the neutral element of $G_1$. Note that its length is $\leq C|h|_P+|v|_{F_X}$, so that we can write it in $F_X$ as a product of $l\leq \delta_{G_1}(C|h|_P+|v|_{F_X})$ conjugates of elements of $R$.
Finally, writing $v=(v\pi(w)^{-1})\pi(w)$ we observe that 
\[ \mathrm{Area}_Q(v)\leq  \mathrm{Area}_Q(\pi(w))+l,\] which combined with Claim \ref{clai:k+pT} implies that
\[ \mathrm{Area}_Q(v)\leq q+k+l.\]
Therefore, 
\[\mathrm{Area}_Q(v)\leq q+\delta_{G_2}(|h|_P)+ \delta_{G_1}(C|h|_P+|v|_{F_X}),\]
and since $q\leq |h|_P$ we are done.
\end{proof}

\begin{corollary}
Let $G_1$, $G_2$ and $Q$ be finitely presented groups, and for $i=1,2$, let $\phi_i:G_i\to Q$ be surjective morphisms. Let $X$ be a finite alphabet and $j_1:F_X\to G_1$ be a surjective morphism. Let $h_n=(g_{n},1)\in P\cap \left(G_1\times 1\right)$, and $v_n\in F_X$ be such that $j_1(v_n)=g_n$. Assume that there is a constant $K \geq 1$ such that
\begin{enumerate}
\item $\frac{1}{K} n\leq |g_{n}|_{G_1},  |v_n|_{F_{X}} \leq K n;$
\item $\mathrm{Area}_Q(v_n)\gtrsim \delta_Q(n/K)$.
\end{enumerate}
Then there is $C\geq 1$ such that $$\delta_Q(n/K)\lesssim \delta_{G_1}(C |h_n|_P)+\delta_{G_2}(C |h_n|_P)+|h_n|_P.$$ 
 \label{cor:DistFibProdImp}
\end{corollary}

\begin{proof} 
Since $P$ projects to $G_1$, we deduce that
\begin{equation}
 |g_{n}|_{G_1}\lesssim |h_n|_P.
 \label{eqnCor1}
\end{equation}
Combining \eqref{eqnCor1} with Assumption (1), we deduce
\begin{equation}
|v_n|_{F_X}\lesssim |h_n|_P.
\label{eqnCor2}
\end{equation}
Applying Proposition \ref{prop:areaDehn} and \eqref{eqnCor2} to $h_n$ and $v_n=v_n(X)$ we obtain that there is a constant $C\geq 1$ such that 
\begin{equation}
\mathrm{Area}_Q(v_n)\leq |h_n|_P + \delta_{G_2}(C|h_n|_P) + \delta_{G_1}(C|h_n|_P).
\label{eqnCor3}
\end{equation}
Since by Assumption (2) $\mathrm{Area}_Q(v_n)\gtrsim \delta_Q(n/K)$, this completes the proof.
\end{proof}

\begin{remark}
Note that linearity of Dehn functions of hyperbolic groups allows us to simplify the conclusion of Corollary \ref{cor:DistFibProdImp} to $\delta_{G_2}(C |h_n|_P)+|h_n|_P\gtrsim \delta_Q(n/K)$ if $G_1$ is hyperbolic.
\label{rmkDistFibProdImp}
\end{remark}

We shall need the following special case of Corollary \ref{cor:DistFibProdImp}, where we assume that $G_1$ is free.

\begin{corollary}\label{cor:DistFibProdFree}
Let $G_1$, $G_2$ and $Q$ be finitely presented groups, and, for $i=1,2$, let $\phi_i:G_i\to Q$ be surjective morphisms. Let $X=X'\sqcup X''$ be a finite alphabet, assume that $G_1=F_{X'}$, and let $j_1$ be the natural projection from $F_X\to F_{X'}$ defined by mapping $X'$ to $X'$ by the identity map and $X''$ to $1$. Let $h_n=(g_{n},1)\in P\cap \left(G_1\times 1\right)$, and $v_n(X')\in F_{X}$ be such that $j_1(v_n)=g_n$. Assume that  there is a constant $K \geq 1$ such that
\begin{enumerate}
\item  $\frac{1}{K} n\leq |v_n|_{F_{X}} \leq K n;$
\item $\mathrm{Area}_Q(v_n)\gtrsim \delta_Q(n/K)$.
\end{enumerate}
Then there is $C\geq 1$ such that $$\delta_Q(n/K)\lesssim \delta_{G_2}(C |h_n|_P)+|h_n|_P.$$ 
\end{corollary}

\section{Proof of Theorem  \ref{thmMain}}

Bridson, Howie, Miller and Short constructed the first examples of subgroups of direct products of free groups which are conilpotent, but not virtually coabelian, that is, do not have a finite index subgroup which is isomorphic to the kernel of a homomorphism from a direct product of free groups to a free abelian group \cite{BriHowMilSho-13}. The class of examples in their work is of rather general nature, but for simplicity we shall restrict ourselves to a specific subfamily of examples that will suffice for our purposes. However, our arguments directly generalize to all of their examples.

For $r\geq 3$ let $F_2^{(1)}, \dots, F_2^{(r)}$ be 2-generated free groups with generating sets $F_2^{(i)}={\mathrm{Free}}(\left\{a_i,b_i\right\})$. Choose finite normal generating sets $Y_i=Y_i(a_i,b_i)$ of the $(r-1)$-th term of the lower central series $\left\langle \left\langle Y_i \right\rangle \right \rangle =\gamma_{r-1}(F_2^{(i)})\unlhd F_2^{(i)}$ and define elements
\[
\begin{array}{ll}
 z_{1,r}=\left(a_1,a_2,\dots,a_r\right), & \hspace{.1cm} z_{2,r}=\left(b_1,b_2,\dots,b_r\right),\\
 z_{3,r}=\left(a_1,a_2^2,\dots,a_r^r\right), & \hspace{.1cm} z_{4,r}=\left(b_1,b_2^2,\dots,b_r^r\right).\\
\end{array}
\]
Denote $Z_r=\left\{z_{1,r},z_{2,r},z_{3,r},z_{4,r}\right\}$ and define the finitely generated subgroup
\[
H_r = \left\langle X_r \right\rangle\leq F_2^{(1)}\times \dots \times F_2^{(r)},
\]
generated by $X_r= Y_1\cup \dots \cup Y_r \cup Z_r$.

Then $H_r$ has the following properties:
\begin{theorem}[{\cite[Section 4]{BriHowMilSho-13}} ]
The group $H_r$ is a finitely presented full subdirect product and $H_r \cap F_2^{(i)}= \gamma_{r-1}(F_2^{(i)})$ for $1\leq i \leq r$.
\end{theorem}

As a consequence the group $H_r$ is a fibre product of the group $F_2^{(1)}=p_1(H_r)$ and the projection $G_{2,r}=p_{2,\dots,r}(H_r)\leq F_2^{(2)}\times \dots \times F_2^{(r)}$ over the 2-generated free nilpotent group $Q_r=F_2^{(1)}/\gamma_{r-1}(F_2^{(1)})$ of class $r-2$. Denote by $\phi_1 :F_2^{(1)} \to Q_r$ and $\phi_2 : G_{2,r} \to Q_r$ the projections defining $H_r$ as a fibre product. Note that we have $\delta_{F_2^{(1)}}(n)\asymp n$ and $\delta_{Q_r}(n)\asymp n^{r-1}$ \cite{BauMilSho-93, Ger-93}. 
The group $G_{2,r}$ is finitely presented by Theorem \ref{thmBHMS13}, since $H_r$ has the VSP property and therefore the same holds for $G_{2,r}$.

We are now ready to state our main result, whose proof will occupy the rest of this section (observe that Theorem \ref{thmMain} is an immediate consequence of Theorem \ref{thmMainExplicit}).
 
\begin{theorem}
 For $r\geq 3$, the Dehn function of the group $G_{2,r}=p_{2,\dots,r}(H_r)\leq F_2^{(2)}\times \dots \times F_2^{(r)}$ satisfies $\delta_{G_{2,r}}(n) \succcurlyeq n^{r-1}$. 
 \label{thmMainExplicit}
\end{theorem}

\begin{proof}
By \cite[Proof of Theorem 5.3]{Ger-93} and \cite{GerHolRil-03}, the $(r-1)$-fold iterated commutators 
\[w_{n}(a_1,b_1)=\left[a_1^n,\left[a_1^n,\dots, \left[a_1^n,b_1^n\right]\dots\right]\right]
\] 
satisfy $\mathrm{Area}_{Q_r}(w_{n}(a_1,b_1))\asymp n^{r-1}\asymp \delta_{Q_r}(n)$. We will require the following auxiliary result.

\begin{lemma}
 There exists $k\geq 1$ such that the following holds: Let $g_{n}\in F_2^{(1)}\cap H_r=\gamma_{r-1}(F_2^{(1)})$ be the element represented by the word $w_{kn}(a_1,b_1)$ and let $h_n=\left(g_{n},1,\dots,1\right)\in H_r$. Then
 \[
 |h_{n}|_{H_r}\lesssim n.
 \]
 \label{lemUndist}
\end{lemma}

\begin{proof}
We use that the group $H_r$ satisfies the VSP property. Thus, there is a finite index subgroup $\Lambda_1\leq F_2^{(1)}$ such that $\Lambda_1 \times \{1\} \leq p_{1,j}(H_r)$ for $2\leq j \leq r$. Note that there is $k\geq 1$ such that $a_1^k,b_1^k\in \Lambda_1$. Thus, we can choose elements $x_{1,j}, y_{1,j}\in H_r$ satisfying
\[
\begin{array}{ll}
 p_1(x_{1,j})=a_1^k, & \hspace{.5cm} p_j(x_{1,j})=1,\\
 p_1(y_{1,j})=b_1^k, & \hspace{.5cm} p_j(y_{1,j})=1,
 \end{array}
\]
for $2\leq j \leq r$.

It is easy to see that the identity
\[
\left[x_{1,2}^n,\left[x_{1,3}^n,\dots,\left[x_{1,r-1}^n,y_{1,r}^n\right]\dots\right]\right]= \left(w_{nk}(a_1,b_1),1,\dots,1\right)=h_{kn}\in H_r
\]
holds. The word $u_n=\left[x_{1,2}^n,\left[x_{1,3}^n,\dots,\left[x_{1,r-1}^n,y_{1,r}^n\right]\dots\right]\right]$ having length $\leq 2^{2(r-1)} n$, this proves the lemma. 
\end{proof}

We are now ready to finish the proof of Theorem \ref{thmMainExplicit}.
We shall apply Corollary \ref{cor:DistFibProdFree}, with:
\begin{itemize}
\item  $X=X_r=X'\sqcup X''$, where $X'=Y_1$ and $X''=Y_2\cup \dots \cup Y_r \cup Z_r$;
\item $P=H_r$, $G_1=F_2^{(1)}$, $G_2=G_{2,r}$ and $Q=Q_r$;
\item $v_n=w_{kn}$.
\end{itemize}
As already observed the words $v_n$ have length $\lesssim n$, but in order to apply Corollary \ref{cor:DistFibProdFree} we need to show that $|v_n|_{F_{X'}}\simeq n$. This immediately results from the following lemma. 

\begin{lemma}
 Let $F_2=\left\langle a, b \right\rangle$ be the 2-generated free group equipped with its standard word length $|\cdot |_{F_2}$ and let $g\in F_2 \setminus \left\langle a \right\rangle$. Then the commutator $\left[a^n,g\right]$ satisfies
 \[
  \left|\left[a^n,g\right]\right|_{F_2}\geq 2n +2.
 \]
\label{lemCommLength}
\end{lemma}
\begin{proof}
 Since $g\notin \left\langle a \right \rangle $ every freely reduced word $w(a,b)$ representing $g$ can be decomposed as $w(a,b)= a^k b^{\pm 1} u(a,b) a^l$ for some $k,l\in \ZZ$ and $u(a,b)$ either the trivial word or such that $b^{\pm 1}u(a,b)$ is a freely reduced word ending in $b^{\pm 1}$. Thus, we have
 \[
  \left[a^n, w(a,b)\right] = a^{n+k} b^{\pm 1} u(a,b) a^{-n} u(a,b)^{-1} b^{\mp 1} a^{-k}
 \]
 and it follows immediately that
 \[
  \left|\left[a^n,g\right]\right|_{F_2}\geq 2 \left(n+k +l(u(a,b))+1\right)\geq 2n+2.
 \]
\end{proof}
So deduce from Corollary \ref{cor:DistFibProdFree} that 
\[n^{r-1}\lesssim \delta_{G_2}(C |h_n|_{H_r})+|h_n|_{H_r}.\]
By Lemma \ref{lemUndist}, there exists $C'\geq 1$ such that 
\[n^{r-1}\lesssim \delta_{G_2}(C' n)+n,\]
so $\delta_{G_2}(n)\succcurlyeq n^{r-1}$, as required.
\end{proof}

\begin{remark}
We want to remark that we do produce a finitely presented full subdirect product of a product of three 2-generated free groups which admits a cubical lower bound on its Dehn function. In particular, this group is virtually the kernel of a homomorphism from a product of three free groups onto a free abelian group, meaning that we obtain a similar result to \cite[Theorem 1.1]{Dis-09}. However, we do not know if our group $G_{2,3}$ is commensurable to the example in \cite{Dis-09}.

More generally, if we could write the groups in \cite{Dis-09} and \cite{LloTes-18} as quotients of finitely presented subgroups of products of four free groups (respectively surface groups) which are conilpotent of class two, then this would give an alternative approach to proving the cubic lower bounds obtained in these works. However, it is not clear that this would significantly simplify the proofs, as it requires the construction of such subgroups and the only known construction of finitely presented conilpotent, non-coabelian subdirect products of free groups is the one of Bridson, Howie, Miller and Short. While their work provides us with some control on the quotients, it does not provide us with the tools to do such a construction.
\end{remark}

\bibliography{References}
\bibliographystyle{amsplain}

\end{document}

%% file: commands.tex
\newcommand{\AAA}{{\mathds A}}
\newcommand{\CC}{{\mathds C}}
\newcommand{\PP}{{\mathbf P}}
\newcommand{\QQ}{{\mathds Q}}
\newcommand{\RR}{{\mathds R}}
\newcommand{\NN}{{\mathds N}}
\newcommand{\ZZ}{{\mathds Z}}
\newcommand{\del}{{\partial}}
\newcommand{\one}{{\mathds {1}}}
\newcommand{\ord}{{\mathcal {O}}}
\newcommand{\ii}{{\mathds {i}}}
\newcommand{\vol}{{\mathrm {vol}}}
\newcommand{\eps}{{\epsilon}}
\def\Mod{{\rm{Mod}}}
\def\C{{\mathds C}}
\def\D{\rm D}
\def\S{\Sigma}
\def\F{{\mathds F}}
\def\FF{\mathcal F}
\def\aut{{\rm{Aut}}}
\def\inn{{\rm{Inn}}}
\def\out{{\rm{Out}}}
\def\isom{{\rm{Isom}}}
\def\mcg{{\rm{MCG}}}
\def\ker{{\rm{ker}}}
\def\im{{\rm{im}}}
\def\dim{{\rm{dim}}}
\def\G{\Gamma}
\def\a{\alpha}
\def\g{\gamma}
\def\L{\Lambda}
\def\Z{{\mathds{Z}}}
\def\H{{\mathds{H}}}
\def\nn{{\bf N}}
\newcommand{\mm}{{\underline{m}}}

\newtheorem{thm}{Theorem}
\newtheorem{clai}[thm]{Claim}

\theoremstyle{plain}
\newtheorem{theorem}{Theorem}[section]
\newtheorem{acknowledgement}[theorem]{Acknowledgement}
\newtheorem{claim}[theorem]{Claim}
\newtheorem{conjecture}[theorem]{Conjecture}
\newtheorem{corollary}[theorem]{Corollary}
\newtheorem{exercise}[theorem]{Exercise}
\newtheorem{lemma}[theorem]{Lemma}
\newtheorem{proposition}[theorem]{Proposition}
\newtheorem{question}{Question}
\newtheorem*{question*}{Question}
\newtheorem{addendum}[theorem]{Addendum}

\theoremstyle{definition}
\newtheorem{remark}[theorem]{Remark}
\newtheorem*{acknowledgements*}{Acknowledgements}
\newtheorem{example}[theorem]{Example}
\newtheorem{definition}[theorem]{Definition}
\newtheorem*{notation*}{Notation}
\newtheorem*{convention*}{Convention}

\renewcommand{\proofname}{Proof}

%% file: LowerBound_LlosaTessera-final.bbl
\providecommand{\bysame}{\leavevmode\hbox to3em{\hrulefill}\thinspace}
\providecommand{\MR}{\relax\ifhmode\unskip\space\fi MR }
\providecommand{\MRhref}[2]{%
  \href{http://www.ams.org/mathscinet-getitem?mr=#1}{#2}
}
\providecommand{\href}[2]{#2}
\begin{thebibliography}{10}

\bibitem{AliBes-06}
E.~Alibegovi\'{c} and M.~Bestvina, \emph{Limit groups are {$\rm CAT(0)$}}, J.
  London Math. Soc. (2) \textbf{74} (2006), no.~1, 259--272.

\bibitem{BauBriMilSho-97}
G.~Baumslag, M.R. Bridson, C.F. Miller~III, and H.~Short, \emph{Finitely
  presented subgroups of automatic groups and their isoperimetric functions},
  Journal of the London Mathematical Society \textbf{56} (1997), no.~2,
  292--304.

\bibitem{BauMilSho-93}
G.~Baumslag, C.F. Miller~III, and H.~Short, \emph{Isoperimetric inequalities
  and the homology of groups}, Inventiones mathematicae \textbf{113} (1993),
  no.~1, 531--560.

\bibitem{BauMyaRem-99}
G.~Baumslag, A.~Myasnikov, and V.~Remeslennikov, \emph{Algebraic geometry over
  groups. {I}. {A}lgebraic sets and ideal theory}, J. Algebra \textbf{219}
  (1999), no.~1, 16--79.

\bibitem{Bie-76}
R.~Bieri, \emph{Homological dimension of discrete groups}, Mathematics
  Department, Queen Mary College, London, 1976, Queen Mary College Mathematics
  Notes.

\bibitem{Bri-99}
M.R. Bridson, \emph{Doubles, finiteness properties of groups, and quadratic
  isoperimetric inequalities}, J. Algebra \textbf{214} (1999), no.~2, 652--667.

\bibitem{Bri-01}
M.R. Bridson, \emph{On the subgroups of semihyperbolic groups}, Essays on geometry
  and related topics, {V}ol. 1, 2, Monogr. Enseign. Math., vol.~38,
  Enseignement Math., Geneva, 2001, pp.~85--111.

\bibitem{Bri-02}
M.R. Bridson, \emph{{The geometry of the word problem}}, Invitations to geometry and
  topology \textbf{7} (2002), 29--91.

\bibitem{BriGru-04}
M.R. Bridson and F.J. Grunewald, \emph{Grothendieck's problems concerning
  profinite completions and representations of groups}, Ann. of Math. (2)
  \textbf{160} (2004), no.~1, 359--373.

\bibitem{BriHowMilSho-02}
M.R. Bridson, J.~Howie, C.F. Miller~III, and H.~Short, \emph{{The subgroups of
  direct products of surface groups}}, Geometriae Dedicata \textbf{92} (2002),
  95--103.

\bibitem{BriHowMilSho-13}
M.R. Bridson, J.~Howie, C.F. Miller~III, and H.~Short, \emph{{On the finite presentation of subdirect products and the nature
  of residually free groups}}, American Journal of Math. \textbf{135} (2013),
  no.~4, 891--933.

\bibitem{BriMil-09}
M.R. Bridson and C.F. Miller~III, \emph{{Structure and finiteness properties of
  subdirect products of groups}}, Proc. Lond. Math. Soc. \textbf{98} (2009),
  no.~3, 631--651.

\bibitem{BurTab-02}
J.~Burillo and J.~Taback, \emph{{Equivalence of geometric and combinatorial
  Dehn functions}}, New York J. Math \textbf{8} (2002), 169--179.

\bibitem{CarFos-17}
W.~Carter and M.~Forester, \emph{{The Dehn functions of Stallings--Bieri
  groups}}, Mathematische Annalen \textbf{368} (2017), no.~1-2, 671--683.

\bibitem{Dis-08-II}
W.~Dison, \emph{{An isoperimetric function for Bestvina--Brady groups}},
  Bulletin of the London Mathematical Society \textbf{40} (2008), no.~3,
  384--394.

\bibitem{Dis-08}
W.~Dison, \emph{{Isoperimetric functions for subdirect products and
  Bestvina-Brady groups}}, Ph.D. thesis, Imperial College London, 2008.

\bibitem{Dis-09}
W.~Dison, \emph{A subgroup of a direct product of free groups whose {D}ehn
  function has a cubic lower bound}, J. Group Theory \textbf{12} (2009), no.~5,
  783--793.

\bibitem{DisEldRilYou-09}
W.~Dison, M.~Elder, T.R. Riley, and R.~Young, \emph{{The Dehn function of
  Stallings' group}}, Geometric and Functional Analysis \textbf{19} (2009),
  no.~2, 406--422.

\bibitem{Ger-93}
S.M. Gersten, \emph{Isoperimetric and isodiametric functions of finite
  presentations}, Geometric group theory \textbf{1} (1993), 79--96.

\bibitem{Ger-95}
S.M. Gersten, \emph{Finiteness properties of asynchronously automatic groups},
  Geometric group theory (Columbus OH, 1992), vol.~3, Ohio State Univ. Math.
  Res. Inst. Publ., deGruyter, Berlin, 1995, pp.~121--133.

\bibitem{GerHolRil-03}
S.M. Gersten, D.F. Holt, and T.R. Riley, \emph{{Isoperimetric inequalities for
  nilpotent groups}}, {Geometric \& Functional Analysis GAFA} \textbf{13}
  (2003), no.~4, 795--814.

\bibitem{Koc-10}
D.H. Kochloukova, \emph{On subdirect products of type {${\rm FP}_m$} of limit
  groups}, J. Group Theory \textbf{13} (2010), no.~1, 1--19.

\bibitem{Kuc-14}
B.~Kuckuck, \emph{Subdirect products of groups and the
  {$n$}-{$(n+1)$}-{$(n+2)$} conjecture}, Q. J. Math. \textbf{65} (2014), no.~4,
  1293--1318.

\bibitem{LloTes-18}
C.~Llosa~Isenrich and R.~Tessera, \emph{{On the Dehn functions of K\"ahler
  groups}}, arXiv:1807.03677 (2018), {to appear in Groups, Geometry, and
  Dynamics}.

\bibitem{Sta-63}
J.R. Stallings, \emph{A finitely presented group whose 3-dimensional integral
  homology is not finitely generated}, Amer. J. Math. \textbf{85} (1963),
  541--543.

\end{thebibliography}
